\documentclass[article,12pt]{amsart}
\usepackage{amsfonts}
\usepackage{dsfont}
\usepackage{amssymb}
\usepackage{datetime}
\usepackage{setspace}
\usepackage{enumerate}
\usepackage{amsthm}
\usepackage{graphicx}
\usepackage{epstopdf}

\newtheorem{theorem}{Theorem}[section]
\newtheorem{proposition}[theorem]{Proposition}
\newtheorem{lemma}[theorem]{Lemma}
\newtheorem{corollary}[theorem]{Corollary}

\theoremstyle{definition}
\newtheorem{example}[theorem]{Example}
\newtheorem{definition}[theorem]{Definition}

\newtheorem{remark} [theorem] {Remark}

\input epsf
\begin{document}

\title{ The dual Radon - Nikodym property \\
for finitely generated Banach $C(K)$-modules}

\author{Arkady Kitover}

\address{Department of Mathematics, Community College of Philadelphia, 1700 Spring Garden St., Philadelphia, PA, USA}

\email{akitover@ccp.edu}

\author{Mehmet Orhon}

\address{Department of Mathematics and Statistics, University of New Hampshire, Durham, NH 03824, USA}

\email{mo@unh.edu}

\subjclass[2010]{Primary 46B20; Secondary 47B22, 46B42}

\date{\today}

\keywords{Radon - Nikodym property, Banach modules, Banach lattices}
\begin{abstract} We extend the well-known criterion of Lotz for the dual Radon-Nikodym property (RNP) of Banach lattices to finitely generated Banach $C(K)$-modules and Banach $C(K)$-modules of finite multiplicity. Namely, we prove that if $X$ is a Banach space from one of these classes then its Banach dual $X^\star$ has the RNP iff $X$ does not contain a closed subspace isomorphic to $\ell^1$.

\end{abstract}
\maketitle

\markboth{Arkady Kitover and Mehmet Orhon}{The dual Radon - Nikodym property}

\section{Introduction}

Let us start with reminding the reader about the following equivalences in the class of Banach lattices.

\begin{theorem} \label{tA} (Lozanovsky~\cite{Loz} - Lotz~\cite{Lot}, see also~\cite[Theorem 2.4.15, p.94]{MN}) Let $X$ be a Banach lattice. Then the following conditions are equivalent.
\begin{enumerate}
  \item $X$ is reflexive.
  \item $X$ does not contain a copy \footnote{If $X$ and $Y$ are Banach spaces we say that $X$ contains a copy of $Y$ if there is a closed subspace of $X$ linearly isomorphic to $Y$.} of either $c_0$ or $\ell^1$.
  \item $X$ does not contain a copy of either $c_0$ or $\ell^1$ as a sublattice. \footnote{If $X$ and $Y$ are Banach lattices we say that $X$ contains a copy of $Y$ as a sublattice if there is a closed sublattice of $X$ lattice isomorphic to $Y$.}
\end{enumerate}

\end{theorem}

\begin{theorem} \label{tB}  (Lozanovsky~\cite{Loz1}, see also~\cite[Theorem 2.5.6, p.104 and Theorem 2.4.12, p. 92]{MN}) Let $X$ be a Banach lattice. Then the following conditions are equivalent.
\begin{enumerate}
  \item $X$ is weakly sequentially complete.
  \item $X$ does not contain a copy of $c_0$.
 \item $X$ does not contain a copy of $c_0$ as a sublattice.
\end{enumerate}

\end{theorem}

Before we remind the reader the next result in this direction let us recall the following definition.

\begin{definition} \label{d7} A Banach space $X$ is said to have the Radon-Nikodym property (RNP) if for every finite measure space $(\Omega, \Sigma, \lambda)$ and for every bounded linear operator $T: L^1(\lambda) \rightarrow X$ there exists a strongly measurable $g \in L^\infty(\lambda, X)$ such that
$$ Tf = \int \limits_{\Omega} {fg d\lambda}, \; f \in L^1(\lambda),$$
where the integral is the Bochner integral.
\end{definition}

\begin{theorem} \label{tC}  (Lotz, see~\cite[Theorem 5.4.14, p.367]{MN}) Let $X$ be a Banach lattice. The following conditions are equivalent.
\begin{enumerate}
  \item The Banach dual $X^\star$ of $X$ has RNP.
  \item $X$ does not contain a copy of $\ell^1$.
\item $X^\star$ does not contain a copy of either $c_0$ or $L^1[0,1]$ as a sublattice.
\end{enumerate}
\end{theorem}

The statements of Theorems~\ref{tA} -~\ref{tC} become false if instead of Banach lattices we consider arbitrary Banach spaces. For Theorems~\ref{tA} and~\ref{tB} a counterexample is provided by the famous James' space~\cite{Ja} while in  the case of Theorem~\ref{tC} we need to use another example of James in~\cite{Ja1} \footnote{The James' space from~\cite{Ja}, being quasi-reflexive, has the dual RNP.} where he constructed a separable Banach space $X$ not containing a copy of $\ell^1$ and such that $X^\star$ is not separable. It follows from a later result of Stegall~\cite{St} that the space $X$ from~\cite{Ja1} does not have the dual RNP.

But, if instead of the class of all Banach spaces we consider the much smaller classes of finitely generated Banach $C(K)$-modules or Banach $C(K)$-modules of finite multiplicity, which while not contained in the class of all Banach lattices can be considered as its nearest relatives, the analogues of Theorems~\ref{tA} -~\ref{tC} remain true.

We refer the reader for precise definitions of the notions used in the next two statements to the next section of the current paper.

\begin{theorem} \label{tD} (\cite[Theorems 1 and 2]{KO}) Let $X$ be a finitely generated Banach $C(K)$-module or a Banach $C(K)$-module of finite multiplicity. Then the following conditions are equivalent.
\begin{enumerate}
  \item $X$ is reflexive.
  \item $X$ does not contain a copy of either $c_0$ or $\ell^1$.
  \item Every cyclic subspace of $X$ does not contain a copy of either $c_0$ or of $\ell^1$.
\end{enumerate}

\end{theorem}

\begin{theorem} \label{tE} (\cite[Theorems 3.1 and 3.8]{KO1}) Let $X$ be a finitely generated Banach $C(K)$-module or a Banach $C(K)$-module of finite multiplicity. Then the following conditions are equivalent.
\begin{enumerate}
  \item $X$ is weakly sequentially complete.
  \item $X$ does not contain a copy of $c_0$
  \item Every cyclic subspace of $X$ does not contain a copy of $c_0$.
\end{enumerate}

\end{theorem}

\begin{remark} \label{r1} Every cyclic subspace of a Banach $C(K)$-module can be endowed in the unique way with the structure of a Banach lattice compatible with its structure of $C(K)$-submodule (see~\cite{KO, KO1} for more details). Therefore we can state that conditions of Theorems~\ref{tD} and~\ref{tE} are equivalent to any cyclic subspace not containing either $c_0$ or $\ell^1$ as a sublattice in the case of Theorem~\ref{tD} and not containing $c_0$ as a sublattice in the case of Theorem~\ref{tE}.

\end{remark}

The goal of the present paper is to show that similar to Theorems~\ref{tD} and~\ref{tE} the result of Theorem~\ref{tC} can be extended to the classes of finitely generated $C(K)$-modules and $C(K)$-modules of finite multiplicity.

We will need the following important characterization of the dual RNP property in Banach spaces.

\begin{theorem} \label{tI} (Stegall - Uhl).Let $X$ be a Banach space. The following conditions are equivalent.

(1) $X^\star$ has RNP.

(2) For every separable subspace $Y$ of $X$ its conjugate $Y^\star$ is also separable.

\end{theorem}

\begin{remark} \label{r5} The implication (1) $\Rightarrow$ (2) was proved by Stegall (see~\cite[Theorem2]{St}) and the implication (2) $\Rightarrow$ (1) by Uhl (see~\cite[Corollary 3]{U})

\end{remark}

At the end of the introduction we want to emphasize that one of the reasons to study the dual RNP is the following important result.

\begin{theorem} \label{tF} (~\cite[Theorem 1, p.98]{DU}) Let $(\Omega, \Sigma, \mu )$ be a finite measure space, $1 \leq p < \infty$, and $X$ be a Banach space. Then $L^p(\mu,X)^\star = L^q(\mu,X^\star)$, where $1/p+1/q=1$, if and only if $X^\star$ has the Radon-Nikodym property with respect to $\mu$.

\end{theorem}

\begin{remark} \label{r2} (1) We remind the reader that $L^p(\mu, X)$, $1 \leq p < \infty$, is the space of (classes) of all Bochner integrable functions on $\Omega$ endowed with the norm  $\|f\| = \big{(} \int \limits_\Omega \|f(\omega)\|_X^p d\mu \big{)}^{1/p}$.

(2) If $X^\star$ does not have the RNP the description of $L^p(\mu,X)^\star$ becomes rather involved. For a detailed discussion of it and for the background of Theorem~\ref{tF} we refer the reader to~\cite{DU}.
\end{remark}

\section{Preliminaries}

\bigskip All the linear spaces will be considered either over the field of
real numbers $\mathds{R}$ or the field of complex numbers $\mathds{C}$. If $X$
is a Banach space we will denote its Banach dual by $X^{\star}$.

Let us recall some definitions.

\begin{definition}
\label{d1} Let $K$ be a compact Hausdorff space and $X$ be a Banach space. We
say that $X$ is a Banach $C(K)$-module if there is a continuous unital algebra
homomorphism $m$ of $C(K)$ into the algebra $L(X)$ of all bounded linear
operators on $X$.
\end{definition}

\begin{remark}
\label{r3} Because $\ker{m}$ is a closed ideal in $C(K)$ by considering, if
needed, $C(\tilde{K}) = C(K)/\ker{m}$ we can and will assume without loss of
generality that $m$ is a contractive homomorphism (see~\cite{KO}) and $\ker{m}
= 0$. Then (see~\cite[Lemma 2 (2)]{HO}) $m$ is an isometry. Moreover, when it
does not cause any ambiguity we will identify $f    \in C(K)$ and $m(f)    \in L(X)$.
\end{remark}

\begin{definition}
\label{d2} Let $X$ be a Banach $C(K)$-module and $x    \in X$. We introduce the
\textit{cyclic} subspace $X(x)$ of $X$ as $X(x) = cl\{fx : f    \in C(K)\}$.
\end{definition}

The following proposition was proved in~\cite{Ve} (see also~\cite{Ra}) in the case
when the compact space $K$ is extremally disconnected and announced for an
arbitrary compact Hausdorff space $K$ in~\cite{Ka}. It follows as a special
case from~\cite[Lemma 2 (2)]{HO}.

\begin{proposition}
\label{p1} Let $X$ be a Banach $C(K)$-module, $x    \in X$, and $X(x)$ be the
corresponding cyclic subspace. Then, endowed with the cone $X(x)_{+} = cl \{fx
: f    \in C(K), f \geq0\}$ and the norm inherited from $X$, $X(x)$ is a Banach
lattice with positive quasi-interior point $x$.
\end{proposition}

Our next proposition follows from Theorem 1 (3) in~\cite{Or}

\begin{proposition}
\label{p2} The center $Z(X(x))$ of the Banach lattice $X(x)$ is isometrically
isomorphic to the weak operator closure of $m(C(K))$ in $L(X(x))$.
\end{proposition}

Now we can introduce one of the two main objects of interest in the current paper.

\begin{definition}
\label{d3} Let $X$ be a Banach $C(K)$-module. We say that $X$ is finitely
generated if there are an $n    \in \mathds{N}$ and $x_{1}, \ldots, x_{n}    \in X$
such that the set $\sum\limits_{i=1}^{n} X(x_{i})$ is dense in $X$.
\end{definition}

\begin{definition}
\label{d4} Let $X$ be a Banach space and $\mathcal{B}$ be a Boolean algebra of
projections on $X$. The algebra $\mathcal{B}$ is called Bade complete
(see~\cite[XVII.3.4, p.2197]{DS} and~\cite[V.3, p.315]{Sch}) if

(1) $\mathcal{B}$ is a complete Boolean algebra.

(2) Let $\{\chi_{\gamma }\}_{\gamma   \in \Gamma }$ be an increasing net in
$\mathcal{B}$, $\chi$ be the supremum of this net, and $x   \in X$. Then the net
$\{\chi_{\gamma }x\}$ converges to $\chi x$ in norm in $X$.
\end{definition}

\begin{definition}
\label{d5} Let $\mathcal{B}$ be a Bade complete Boolean algebra of projections
on $X$. $\mathcal{B}$ is said to be of \textit{uniform multiplicity} $n,$ if
there exist a set of nonzero pairwise disjoint idempotents $\{e_{\alpha }\}$ in
$\mathcal{B}$ with $\sup e_{\alpha }=1$ such that for any $e_{\alpha }$ and for
any $e   \in \mathcal{B}$, $e\leq e_{\alpha }$ the $C(K)$-module $eX$ has exactly
$n$ generators.
\end{definition}

We will need the following result of Rall~\cite{Ra1} (for a proof, see Lemma 2
in~\cite{Or1}).

\begin{lemma}
\label{l3} Let $\mathcal{B}$ be of uniform multiplicity one on $X$. Then $X$
may be represented as a Banach lattice with order continuous norm such that
$\mathcal{B}$ is the Boolean algebra of band projections on $X$.
\end{lemma}

\begin{definition}
\label{d6} A Bade complete Boolean algebra of projections $\mathcal{B}$ on $X$
is said to be of finite multiplicity on $X$ if there exists a collection of
disjoint idempotents $\{e_{\alpha }\}$ in $\mathcal{B}$ such that, for each
$\alpha,$ $e_{\alpha }X$ is $n_{\alpha }$-generated and $\sup e_{\alpha }=1.$
\end{definition}

\begin{remark} \label{r4}
  The collection $\{n_{\alpha }\}$ of positive integers in Definition~\ref{d6} need not be
bounded.
\end{remark}
\begin{theorem} \label{tJ}  (Bade \cite[XVIII.3.8, p. 2267]{DS}). Let $X$ be a Banach $C(K)$-module of finite multiplicity.
Then there exists a sequence of disjoint idempotents $\{e_{n}\}$ in $\mathcal{B}$
such that, for each $n,$ $\mathcal{B}$ is of uniform multiplicity $n$ on
$e_{n}X$ and $\sup e_{n}=1.$ Also the norm closure of the sum of the sequence
of the spaces $\{e_{n}X\}$ is equal to $X.$
\end{theorem}

\section{The main Results}

\bigskip
\begin{theorem}
\label{t1} Let $X$ be a finitely generated Banach $C(K)$-module. Then the
following are equivalent:\newline(1) $X^\star$ has the Radon-Nikodym
property.\newline(2) $X$ does not contain any copy of $\ell^1$.
\end{theorem}

\begin{proof}
The implication $(1)   \Rightarrow   (2)$ follows from Theorem~\ref{tI}. It can also be proved in the following way.

Suppose (1) holds. Then it is well known that $X^\star$ does not contain
any copy of $L^{1}(0,1)$ (~\cite[5.4.4(ii), p.363 and 5.4.8, p.365]{MN} ).
Then by a result of Hagler~\cite{Ha}, we have that $X$ does not contain a copy of
$\ell^1$. Hence (1) implies (2) for all Banach spaces $X$.

Conversely, suppose that (2) holds. We will complete the proof by induction on
the number of generators of $X$. Let $X$ have one generator. That is $X$ is
cyclic. Then it follows from Proposition~\ref{p1} and Theorem~\ref{tC} that $X^\star$ has the RNP.  Now suppose that whenever $X$ has $r$
generators for some fixed $r\geq1$ and satisfies (2), then $X^\star$ has
the RNP. Let $X$ have $r+1$ generators and satisfy (2).
Suppose $\{x_{0},x_{1},\ldots,x_{r}\}$ is a set of generators of $X$. Let
$Y=X(x_{1},x_{2},\ldots,x_{r})$. Then $Y$ is a Banach $C(K)$-module with $r$
generators and as a subspace of $X$, it satisfies (2). Hence, by the induction
hypothesis, $Y^\star$ has the RNP. Now consider the
cyclic space \ $X/Y=X/Y([x_{0}])$ where $[x_{0}]=x_{0}+Y$. We have that
$\left(  X/Y\right)  ^\star = Y^{o}   \subset X^\star$ where $Y^{o}$ is the
polar (annihilator) of $Y$ in $X^\star$. Since $X$ satisfies (2), by
Hagler's Theorem~\cite{Ha}, $X^\star$ does not contain a copy of $L^{1}(0,1)$
and the same is true for $(X/Y)^\star$ as a subspace of $X^\star$. Again by Hagler's
Theorem, this means that $X/Y$ satisfies (2).  Since $X/Y$ is a cyclic Banach
space, it is representable as a Banach lattice. Therefore, by Lotz's Theorem~\ref{tC},
$Y_0 = (X/Y)^\star$ has the RNP. Recall also that $ X^\star/Y^0 = Y^\star$. Since the RNP is a three space property~\cite[6.5.b, p. 202]{CG}, we
have that $X^\star$ has the Radon-Nikodym property.
\end{proof}

It is possible to sharpen the result of Theorem~\ref{t1} if we assume that the cyclic subspaces of $X$
have order continuous norm when represented as Banach lattices. Initially, we need a result which guarantees that
a Banach lattice $X$ which does not contain a copy of $\ell^1$ as a sublattice
does not contain a copy of $\ell^1$ as a subspace. This is not true in general.
For example, $\ell^1$ is not isomorphic to any sublattice of $\ell^{\infty}$,
however any separable Banach space (hence in particular $\ell^1$) may be
embedded isometrically into $\ell^{\infty}$.

To proceed we need to recall the following important result  by Lotz and Rosenthal.

\begin{theorem} \label{tH} (Lotz and Rosenthal~\cite{LR}, see also~\cite[ Theorem 5.2.15, p.
345]{MN})

Let $X$ be a Banach lattice. The following conditions are equivalent.

(1) For each $x \in X_+$ the order interval $[0,x]$ is weakly sequentially precompact.

(2) $X^\star$ does not contain any copy of $L^1[0,1]$ as a sublattice.

\end{theorem}

The Lotz - Rosenthal theorem leads to the following lemma.

\begin{lemma}
\label{l1} Let $X$ be a Banach lattice with order continuous norm then $X$
contains a copy of $\ell^1$ if and only if $X$ contains a copy of $\ell^1$ as a sublattice.
\end{lemma}

\begin{proof}
Suppose $X$ does not contain a copy of $\ell^1$ as a sublattice. Then
$X^\star$ does not contain any copy of $c_{0}$ as a sublattice~\cite
[Proposition 2.3.12, p. 83]{MN}. Also since $X$ has order continuous norm any order
interval in $X$ is weakly compact~\cite[Theorem 2.4.2, p. 86]{MN}. Hence by
the Eberlein - \v{S}mulian theorem any order interval in $X$ is weakly sequentially
compact. Then, by Theorem~\ref{tH} of Lotz and Rosenthal $X^\star$ does not contain any copy of $L^{1}(0,1)$ as a sublattice.
Finally, by the equivalence (2) $\Leftrightarrow$ (3) in Theorem~\ref{tC} by Lotz we conclude that $X$ does not contain any copy of $\ell^1$.
\end{proof}

\begin{theorem} \label{t2}
 Let $K$ be a hyperstonian compact space and let $X$ be a finitely generated Banach
$C(K)$-module such that the algebra $\mathcal{B}$ of the idempotents in $C(K)$, is a Bade
complete Boolean algebra of projections on $X$. Then the following conditions are
equivalent
\begin{enumerate}
  \item $X^\star$ has the Radon-Nikodym property.
  \item $X$ does not contain any copy of $\ell^1$.
  \item No cyclic subspace of $X$ contains a copy of $\ell^1$.
  \item No cyclic subspace of $X$, when represented as a Banach lattice, contains a copy of $\ell^1$ as a sublattice.
\end{enumerate}
\end{theorem}

\begin{proof}
It is clear that (2)$   \Rightarrow   $(3). Let us show that (3)$   \Rightarrow   $(2). We will use induction on the number of generators of
$X$. When $n=1$, $X$ is cyclic, so (3)$   \Rightarrow   $(2) trivially. Suppose for
some $r\geq1,$ for all $X$ with $r$ generators we have (3)$   \Rightarrow   $(2).
Suppose $X$ has $r+1$ generators $\{x_{0},x_{1},\ldots,x_{r}\}$ and satisfies
(3). Let $Y=X(x_{1},\ldots,x_{r})$. Then $Y$ satisfies (2) by the induction
hypothesis. On the other hand, $X/Y$ $=X/Y([x_{0}])$ where $[x_{0}]=x_{0}+Y.$
Since $\mathcal{B}$ is Bade complete on $X,$ by~\cite[Lemma 1]{KO}, $\mathcal{B}$
is Bade complete on the cyclic space $X/Y$. By Lemma~\ref{l3}, $X/Y$ can be represented as a
Banach lattice with order continuous norm and quasi-interior point $[x_{0}]$
such that the algebra $\mathcal{B}$ is isometrically isomorphic to the algebra of band projections on the Banach
lattice $X/Y$. Now (3) implies that $\ell^1$ is not
contained in $X/Y$ as a sublattice (see the proof of Theorem 1 in~\cite[p. 480]{KO}). Then by
Lemma~\ref{l1}, we have that $X/Y$ does not contain a copy of $l^{1}.$ Since not
containing $\ell^1$ is a three space property~\cite[Theorem 3.2.d, p.96]{CG}, we
have that $X$ does not contain any copy of $\ell^1$.

The implication (3) $\Rightarrow$ (4) is trivial.

Assume (4). Every cyclic subspace of $X$ when represented as a Banach lattice has order continuous norm. Thus the implication (4) $\Rightarrow$ (3) follows from Lemma~\ref{l1}. It completes the proof.

\end{proof}

The purpose of the next two examples is to show that condition $(3)$ in Theorem~\ref{t2} cannot be weakened as follows. Let $x_1, \ldots , x_n$ be some system of generators of $X$. Assume that none of the cyclic subspaces $X(x_i), i= 1, \ldots , n$ contains a copy of $\ell^1$. Then $X^\star$ has RNP. (Compare Examples 1 and 2 on page 482 in~\cite{KO})

\begin{example} \label{e1} Let $X = L^1(0,1) \oplus L^2(0,1)$ considered as a $L^\infty(0,1)$ Banach module. Let $x_1=(\mathbf{0}, \mathbf{1})$ and $x_2= (-\mathbf{1}, \mathbf{1})$. Then $\{x_1, x_2\}$ is a system of generators of $X$. Clearly $X(x_1) = L^2$ and $X(x_2)$ is isomorphic to $L^2$ while $X^\star$ does not have RNP.
\end{example}

In Example~\ref{e1}, $X$ is a non-atomic Banach lattice. It is easy to provide a similar example when $X$ is a discrete BL.

\begin{example} \label{e2} Let $w_n$ be an increasing sequence of positive real numbers such that $\sum \limits_{n=1}^\infty \frac{1}{w_n}<\infty$. We consider the Hilbert space
$$ \ell^2(w_n) = \{\{a_n\}: \sum \limits_{n=1}^\infty |a_n|^2w_n^2 < \infty \}. $$
with the norm $\|\{a_n\}\| = \sqrt{\sum \limits_{n=1}^\infty |a_n|^2w_n^2}$. Then $\ell^2(w_n) \subset \ell^1$. Let $X= \ell^1 \oplus \ell^2(w_n)$ considered as a $\ell^\infty$-module. Let $x_1 = (\mathbf{0}, \{\frac{1}{n w_n}\})$ and $x_2 = (\{-\frac{1}{n w_n}\}, \{\frac{1}{n w_n}\})$. Then, as in Example~\ref{e1} we see that the system $\{x_1, x_2\}$ generates $X$ and that the cyclic subspaces $X(x_1)$ and $X(x_2)$ are reflexive. However $X^\star$ does not have RNP.
\end{example}

\begin{remark} \label{r6} Already mentioned example of James in~\cite{Ja1} shows that, in general, the condition that $X$ is a finitely generated $C(K)$-module cannot be dropped. Indeed, every Banach space can be considered as a Banach module over $\mathds{C}$. Still there is a large class of Banach $C(K)$-modules which are in general not finitely generated but for which the conclusion of Theorem~\ref{t2} remains valid.
\end{remark}
We proceed now to extend the result of Theorem~\ref{t2} to Banach $C(K)$-modules of finite multiplicity.

\begin{lemma} \label{l4} Let $X$ be a Banach $C(K)$-module of uniform multiplicity $n$ and let $x \in X$. Let $\{e_\alpha\}$ be the system of pairwise disjoint idempotents from Definition~\ref{d5}. Then the set $\{\alpha : e_\alpha x \neq 0\}$ is at most countable.

\end{lemma}

\begin{proof} Consider the cyclic subspace $Y$ of $X$ generated by $x$. By Lemma~\ref{l3} $Y$ can be represented as a Banach lattice with order continuous norm. Assume, contrary to our claim, that the set $\{\alpha : e_\alpha x \neq 0\}$ is uncountable. Then for some $k \in \mathds{N}$ there are $e_1, e_2, \ldots \in \{e_\alpha\}$ such that $\|e_n x\| \geq 1/k, n \in \mathds{N}$. Notice that the elements $e_n x$ are pairwise disjoint in the Banach lattice $Y$. Then, by Proposition 0.5.5 in~\cite[ page 36]{Wn} $Y$ contains $\ell^\infty$ as a sublattice. But then by the well known result of Lozanovsky~\cite{Loz2} (see also~\cite[Theorem 4, page 295]{KA}) $Y$ cannot have order continuous norm, a contradiction.

\end{proof}

\begin{corollary} \label{c1} Let $X$ be a Banach $C(K)$-module of uniform multiplicity $n$ and let $Y$ be a separable closed subspace of $X$. Then the set $\{\alpha : e_\alpha Y \neq 0\}$ is at most countable.

\end{corollary}

The proof of the following lemma repeats verbatim the proof of statement $(\star)$ on page 485 in~\cite{KO}).

\begin{lemma} \label{l5} Let $X$ be a Banach $C(K)$-module of finite multiplicity such that any cyclic subspace of $X$ does not contain a copy of $\ell^1$.  Suppose that $\{e_n\}$ is a system of pairwise idempotents in $C(K)$ such that $\sup \limits_n e_n = 1$. Let $\chi_n = e_1 + \ldots + e_n$. Then for any $f \in X^\star$ we have $\|f - \chi_n^\star f\| \rightarrow 0$.

\end{lemma}

\begin{lemma} \label{l2} Let $X$ be a Banach $C(K)$-module of uniform multiplicity $n$. Assume that no cyclic subspace of $X$, represented as a Banach lattice, contains $\ell^1$ as a sublattice. Then $X^\star$ has RNP.

\end{lemma}

\begin{proof} By Theorem~\ref{tI} it is enough to prove that for any closed separable subspace $Y$ of $X$ the conjugate $Y^\star$ is also separable. Thus, let $Y$ be a closed separable subspace of $X$. By Corollary~\ref{c1} there is an at most countable subset $\{e_i, i \in \mathds{N}\}$ of $\{e_\alpha\}$ such that $e_i Y \neq 0, i \in \mathds{N}$ and $e_\alpha Y  =0$ for any $\alpha$ not in this subset. Then $Y \subseteq Z \subseteq X$ where $Z$ is the closure in $X$ of the direct sum $\sum \limits_{i} e_i Y$ (we remind the reader that the idempotents $e_i$ are pairwise disjoint in $C(K)$). For every $i$, $e_iX$ is a finitely generated Banach $e_iC(K)$-module (with $n$ generators)  and by Theorem~\ref{t2} its conjugate $(e_i X)^\star$ has RNP. Therefore by Theorem~\ref{tI} the conjugate $(e_iY)^\star$ is separable.

Before we go to next step in the proof let us notice that for any $i \in \mathds{N}$ we have $(e_i Y)^\star = e_i^\star Z^\star$. Indeed, $(e_i Y)^\star = Z^\star /(e_i Y)^0$ and $e_i$ is a projection on $Z$ with $e_i Z = cl(e_i Y)$. Consequently, $e_i^\star$ is a projection on $Z^\star$ and $(1-e_i^\star)Z^\star = (e_iY)^0$. Therefore $(e_i Y)^\star = e_i^\star Z^\star$.
 Because no cyclic subspace of $X$ contains a copy of $\ell^1$, Lemma~\ref{l5} guarantees that $Z^\star = cl \sum \limits_{i}  (e_i Y)^\star$.
 Indeed, let $f \in Z^\star$, and $m \in \mathds{N}$. Then $\chi_m f \in \sum \limits_{i=1}^m (e_iY)^\star$ and $\chi_m f$ converges to $f$ by norm.
  Thus $Z^\star$ is separable. Then $Y^\star$, being a factor of $Z^\star$ by $Y^0$, the annihilator of $Y$ in $Z^\star$, is separable as well.

\end{proof}

\begin{theorem} \label{t3} Let $X$ be a Banach $C(K)$-module of finite multiplicity. Then the following conditions are equivalent.

(1) $X^\star$ has RNP.

(2) $X$ does not contain a copy of $\ell^1$.

(3) Any cyclic subspace of $X$ does not contain a copy of $\ell^1$.

(4) Any cyclic subspace of $X$ represented as a Banach lattice does not contain $\ell^1$ as a sublattice.

\end{theorem}

\begin{proof} The implication $(1)   \Rightarrow   (2)$ has been already established in the proof of Theorem~\ref{t1}.

The implications $(2)   \Rightarrow   (3)   \Rightarrow   (4)$ are trivial.

The implication $(4) \Rightarrow (3)$ follows from Lemma~\ref{l1}.

It remains to prove that $(3)   \Rightarrow   (1)$. Assume (3). Let $Y$ be a closed separable subspace of $X$. Let $e_n$ be the idempotents from Theorem~\ref{tJ} and let $Z= cl \sum \limits_{n}  e_nY$. Then $Z$ is a separable closed subspace of $X$ and, applying Theorem~\ref{tJ} we see that $Y   \subseteq Z$. Because any cyclic subspace of $X$ does not contain a copy of $\ell^1$ by Lemma~\ref{l5}, as explained in the proof of Lemma~\ref{l2}, $Z^\star = cl \sum \limits_{n}  (e_n Y)^\star$.

Next notice that by Lemma~\ref{l2} the space $(e_nX)^\star$ has RNP and by Theorem~\ref{tI} the space $(e_nY)^\star$ is separable. Therefore $Z^\star$ is separable, $Y^\star$ is separable as a factor of $Z^\star$, and $X^\star$ has RNP by Theorem~\ref{tI}.

\end{proof}

\begin{remark} \label{r3} A slight modification of Example 4.2(3) on page 752 in~\cite{KO1} ( where one replaces $\ell^p$ by $c_0$) provides an example of a Banach $C(K)$-module $X$ with the following properties.
\begin{enumerate}
  \item $X$ is of uniform multiplicity $n$, $n > 1$.
  \item $X$ is not separable.
    \item Every cyclic subspace of $X$ is separable and has separable dual. In particular, $X$ cannot be finitely (or even countably) generated.
        \item There are cyclic subspaces of $X$ that are not weakly sequentially complete.
\end{enumerate}
Thus, while Theorem~\ref{t2} cannot be applied, by Theorem~\ref{t3} $X$ has dual RNP.

\end{remark}

\end{document}